\documentclass[11pt,a4paper,reqno]{amsart}
\usepackage{amsmath,amsfonts,amssymb,fullpage,tabls,graphicx,color}
\usepackage{pdfsync}
\usepackage{ifpdf}
\ifpdf
\fi

\usepackage{textcomp}
\usepackage{
}

\usepackage{mparhack}
\usepackage[normalem]{ulem}
\usepackage{cancel} 

\usepackage{color}
\definecolor{LightYellow}{cmyk}{0,0,0.4,0}
\definecolor{DarkGreen}{rgb}{0,0.5,0}

\newtheorem{teo}{Theorem}[section]
\newtheorem{prop}[teo]{Proposition}

\newtheorem{lemma}[teo]{Lemma}

\newtheorem{example}[teo]{Example}

\newtheorem{oss}[teo]{Remark}

\newcommand{\de}{\partial}
\newcommand{\om}{\omega}
\newcommand{\Om}{\Omega}

\newcommand{\e}{\varepsilon}

\newcommand{\g}{\gamma}
\newcommand{\gp}{\dot{\gamma}}

\newcommand{\strip}{{\mathcal S}}

\renewcommand{\bowtie}{\mathcal W}
\newcommand{\pino}{\mathcal P}

\newcommand{\mdiv}{\mathop{\mathrm{div}}}

\newcommand{\difsim}{\Delta\,}

\DeclareMathOperator*{\dist}{dist}

\DeclareMathOperator*{\Lip}{Lip}
\DeclareMathOperator*{\reach}{\mathcal R}

\newcommand{\R}{{\mathbb R}}

\newcommand{\N}{{\mathbb N}}

\newcommand{\cJ}{{\mathcal J}}

\newcommand{\Hau}{{\mathcal H}}

\newcommand{\Int}{\mathop{Int}}

\newcommand{\impl}{\mathop{\Rightarrow}}

\newcommand{\ds}{\displaystyle}

\title{An overview on the Cheeger problem}

\author[G.P.~Leonardi]{Gian Paolo Leonardi}
\address{Dipartimento di Scienze Fisiche, Informatiche e Matematiche, Universit{\`a} di Modena e Reggio Emilia, Via Campi 213/b, 41100
    Modena, Italy}
\email{gianpaolo.leonardi@unimore.it}

\keywords{Cheeger sets, prescribed mean curvature, isoperimetric}
\subjclass{
49Q10
, 53A10
, 35P15
}

\date{\today}

\begin{document}
\maketitle

\section{Introduction}

Let us fix a bounded open set $\Om\subset \R^{n}$, with $n\geq 2$. Given a Borel set $F\subset \R^{n}$ we denote by $|F|$ its Lebesgue measure (from now on, the \textit{volume} of $F$) and by $P(F)$ its \textit{perimeter} (see section \ref{section:tre} for the definition of the perimeter functional). Then we define the \textit{Cheeger constant of $\Om$} as
\begin{equation}\label{Cheegerconst}
  h(\Om) := \inf \left\{\frac{P(F)}{|F|}\,:\ F\subset \Om,\ |F|>0\right\}\,.
\end{equation}
Any set $E\subset \Om$ such that $\frac{P(E)}{|E|} = h(\Om)$ is called a \textit{Cheeger set} of $\Om$. 
We shall generically refer to the \textit{Cheeger problem}, as far as the computation or estimation of $h(\Om)$, or the characterization of Cheeger sets of $\Om$, are concerned. As we will see later on, the Cheeger problem is deeply connected to other variational problems, ranging from eigenvalue estimates to capillarity models, and even to image segmentation techniques.

The purpose of this note is twofold. First, in order to provide some motivations to the reader, we shall briefly review three relevant problems that show a close connection to the Cheeger problem. Second, after some essential definitions and basic results we give an account of known facts  about the Cheeger problem, as well as of some more recent results obtained by A. Pratelli and the author in \cite{LeoPra2014}. Some key examples are presented in the final section. We also address the interested reader to \cite{ButCarCom2007,CarComPey2009,CasChaMolNov2008,CasFacMei2009,CasMirNov2010,IonLac2005,Grieser2006,Strang2010}, where further applications, developments and extensions of the Cheeger problem are considered.

\section{Some motivations}In this section we synthetically describe three variational problems that are closely connected with the Cheeger problem. 

\subsection{Estimating the smallest eigenvalue of the Laplacian}
The historical motivation of the Cheeger problem is an isoperimetric-type inequality that was first proved by J. Cheeger in \cite{Cheeger1970} in the context of compact, $n$-dimensional Riemannian manifolds without boundary. As a consequence, one obtains the validity of a Poincar\'e inequality with optimal constant uniformly bounded from below by a geometric constant. Let $\lambda_{2}(M)$ be the least non-zero eigenvalue of the Laplace-Beltrami operator on $M$, then Cheeger proved that
\begin{equation}\label{CheegerM}
\lambda_{2}(M) \geq \inf_{A\subset\subset M}\ \frac{P(A)^{2}}{4\min\{V(A),V(M\setminus A)\}^{2}}\,,
\end{equation}
where $V(A)$ and $P(A)$ denote, respectively, the Riemannian volume and perimeter of $A$. 
Here we skip the discussion of the problem on Riemannian manifolds and consider the analogous problem for the $p$-Laplacian ($1\leq p <\infty$) with Dirichlet boundary conditions, with $M$ replaced by a bounded open set $\Om\subset \R^{n}$. To be more specific, we assume that $\Om$ coincides with its essential interior, i.e., that it contains all points $x\in\R^{n}$ for which there exists $r>0$ such that $|B(x,r)\setminus \Om| = 0$. Under this assumption, all (slightly) different definitions of the Cheeger constant, that have been proposed or considered in previous works, actually agree. 
We thus exclude from our analysis domains (like, for instance, a planar open disc minus a diameter) which from the point of view of the Lebesgue measure (and of the perimeter) are not distinguishable from their essential interiors. By approximation (see Theorem \ref{teo:hcontinua}) it will then be possible to deduce estimates that are valid for more general domains (and for the more ``classical'' definition of Cheeger constant, i.e. the minimization of the ratio $\frac{P(F)}{|F|}$ among relatively compact subdomains $F\subset\subset \Om$).

Let $\lambda_{p}(\Om)$ denote the smallest ``eigenvalue'' of the $p$-Laplacian with Dirichlet boundary conditions, for $1\leq p<\infty$:
\[
\lambda_{p}(\Om) := \inf_{u\in W^{1,p}_{0}(\Om)}\frac{\|\nabla u\|_{p}^{p}}{\|u\|_{p}^{p}}\,.
\]
Arguing as in Cheeger's paper (see \cite{LefWei1997,KawFri2003}) one can easily show that   
\begin{equation}\label{Cheegerp}
\lambda_{p}(\Om) \geq \frac{h(\Om)^{p}}{p^{p}}\,,
\end{equation}
where $h(\Om)$ is defined in \eqref{Cheegerconst}. 
The proof of \eqref{Cheegerp} goes as follows: take $u\in W^{1,p}_{0}(\Om)$ with a positive Sobolev norm, and set $q = \frac{p}{p-1}$. By noting that $p/q = p-1$ and thanks to H\"older's inequality, one finds
\begin{align}\label{stimaRayleigh}
\frac{\ds\int|\nabla u|^{p}}{\ds\int |u|^{p}} &\geq \frac{\ds\left(\int |u|^{p-1}|\nabla u|\right)^{p}}{\ds\left(\int |u|^{p}\right)^{p}} = \frac{\ds\left(\int |\nabla |u|^{p}|\right)^{p}}{p^{p}\ds\left(\int |u|^{p}\right)^{p}}\,.
\end{align}
Setting $f = |u|^{p}$, by coarea formula (see \cite{AmbFusPal2000}) one gets
\begin{align}
\label{coareapplic}
\nonumber\int |\nabla f| &= \int_{0}^{+\infty} \frac{P(\{f>t\})}{|\{f>t\}|}\cdot |\{f>t\}|\, dt \geq h(\Om) \cdot \int_{0}^{+\infty}|\{f>t\}|\, dt\\ 
&= h(\Om)\cdot \int f\,,
\end{align}
then by \eqref{coareapplic} one deduces that
\[
\frac{\ds\left(\int |\nabla |u|^{p}|\right)}{\ds\left(\int |u|^{p}\right)} = \frac{\ds\int |\nabla f| }{\ds\int f} \geq h(\Om)\,.
\]
Therefore, \eqref{Cheegerp} follows from this last inequality combined with \eqref{stimaRayleigh}. 

\begin{oss}\rm
We note that, as $p\to 1$, the left-hand side of \eqref{Cheegerp} tends to $\lambda_{1}(\Om)$ while the right-hand side tends to $h(\Om)$. Moreover, we have
\begin{equation}\label{lamacca}
\lambda_{1}(\Om) = h(\Om)\,,
\end{equation}
which means that \eqref{Cheegerp} becomes sharp as $p\to 1$. 
Proving \eqref{lamacca} amounts to show that $\lambda_{1}(\Om) \leq h(\Om)$, as the other inequality directly follows from \eqref{Cheegerp}. To this aim, one can exploit \eqref{coareapplic} on a function $f\in W^{1,1}_{0}(\Om)$ that suitably approximates the characteristic function of a set of finite perimeter $F\subset \Om$, for which $\frac{P(F)}{|F|} \simeq h(\Om)$. To this aim, it is not restrictive to assume that $F$ is relatively compact in $\Om$ and that $\de F$ is smooth, hence $f$ can be defined as a standard regularization of $\chi_{F}$, in such a way that $0\leq f \leq 1$, $|F| \simeq \ds\int f$ and $P(F) \simeq \ds\int|\nabla f|$. In conclusion one obtains
\[
\int|\nabla f| \simeq \frac{P(F)}{|F|} \int f \simeq h(\Om) \int f\,,
\]
which implies \eqref{lamacca}.
\end{oss}

\subsection{Existence of graphs with prescribed mean curvature}
Let $\Om\subset \R^{n}$ be a bounded open domain with Lipschitz boundary. The prescribed mean curvature equation is a nonlinear elliptic partial differential equation of the form
\begin{equation}\label{PMC}
\mdiv \left(\frac{\nabla u}{\sqrt{1+|\nabla u|^{2}}}\right) = H(x)\,,
\end{equation}
where $H$ is a given function on $\Om$. For the moment we do not specify any further property of $H$ and $u$. It is well-known that the left-hand side of \eqref{PMC} represents the scalar mean curvature of the graph $t=u(x)$, up to a division by $n-1$. The prescribed mean curvature equation arises as the Euler-Lagrange equation of the functional
\begin{align}\label{functionalPMC}
\cJ[u] = \int_{\Om} \sqrt{1+|\nabla u(x)|^{2}}\, dx + & \int_{\Om} H(x) u(x)\, dx \\\nonumber & + \int_{\de \Om} |u(y)-\varphi(y)|\, d\Hau^{n-1}(y)
\end{align}
where $\Hau^{n-1}$ denotes the Hausdorff $(n-1)$-dimensional measure in $\R^{n}$ and $\varphi\in L^{1}(\de\Om)$ is a boundary datum. The minimization of \eqref{functionalPMC} corresponds to the physical problem of finding the stable equilibrium configurations for a fluid-gas interface in a cylindrical tube of cross-section $\Om$, subject to surface tension, bulk forces, and boundary conditions. In \cite{Giusti1978} (see also \cite{Giaquinta1974}) some conditions for the existence and uniqueness of solutions to \eqref{PMC} (even without specifying boundary conditions) are found. In particular, we have the following result:
\begin{teo}[\cite{Giusti1978}]\label{teo:Giusti}
Let $\Om$ be a bounded Lipschitz domain, and let $H\in \Lip(\Om)$. Then, the equation \eqref{PMC} admits at least a solution $u\in C^{2}(\Om)$ if and only if
\begin{equation}\label{NCGiusti}
\left|\int_{A} H(x)\, dx\right| < P(A)
\end{equation}
for all $A\subset \Om$ with $0<|A|<|\Om|$. If, in addition, $\left|\int_{\Om} H(x)\right| = P(\Om)$, then the solution $u$ to \eqref{PMC} is \textit{unique} up to additive constants and has a ``vertical contact'' at $\de \Om$.
\end{teo}
The proof of Theorem \ref{teo:Giusti} uses a straightforward application of the divergence theorem for the ``only if'' part, while becomes more technical in the ``if'' part. The case $|\int_{\Om}H(x)| < P(\Om)$ is easier and can be handled by showing the existence of smooth minimizers of the functional $\cJ[u]$ defined in \eqref{functionalPMC}. The critical case $|\int_{\Om}H(x)| = P(\Om)$ is more subtle and requires the notion of generalized solution of \eqref{PMC} in the sense of Miranda \cite{Miranda1964}. We refer to \cite{Giusti1978} for more details.

In order to better exploit the link between the existence of solutions to \eqref{PMC} and the Cheeger problem, we focus on the case $H(x) = H$ constant. We also assume without loss of generality that $H\geq 0$. Then, Theorem \ref{teo:Giusti} implies that a solution $u\in C^{2}(\Om)$ to the constant mean curvature equation
\begin{equation}\label{CMC}
\mdiv \left(\frac{\nabla u(x)}{\sqrt{1+|\nabla u(x)|^{2}}}\right) = H
\end{equation}
exists if and only if $H\leq h(\Om)$ and no proper subset $A$ of $\Om$ is Cheeger in $\Om$. In this sense, the Cheeger constant provides a threshold for the prescribed mean curvature, in order that a solution to \eqref{CMC} may exist. A particularly interesting situation occurs in the limit case $H = h(\Om)$ and when $\Om$ is \textit{uniquely self-Cheeger}, in the sense that $\Om$ is Cheeger in itself and no other proper subset of $\Om$ is Cheeger in $\Om$. Indeed, in this case one gains not only existence but also \textit{uniqueness} (up to a vertical translation) of the solution to \eqref{CMC}. This much more rigid situation corresponds to the case of a graph with constant mean curvature $H=h(\Om)$, that meets the boundary of the cylinder $\Om\times \R$ in a tangential way (thus, the gradient $\nabla u(x)$ blows up as $x$ tends to $\de\Om$) and whose geometrical shape is, therefore, uniquely determined up to a translation. The physical interest for these optimal shapes becomes immediately apparent: indeed, they represent the equilibrium configurations of the capillary free-surfaces formed by perfectly wetting fluids inside a cylindrical container of cross-section $\Om$ under zero gravity conditions. 

\subsection{Stable shapes for Total Variation minimization}
In \cite{RudOshFat1992} (see also the analysis performed in \cite{ChaLio1997}) a variational method, now called \textit{ROF model}, was proposed for the regularization of noisy images. Let $g\in L^{2}(\R^{2})$ be a given \textit{image} to be regularized. The idea is to preserve the essential contours and textures of the objects depicted in the image, while removing noise. To this aim, one can solve the following variational problem:
\begin{equation}\label{ROFproblem}
\min_{u\in L^{2}(\R^{n})\cap BV(\R^{n})}\ \int_{\R^{n}} |Du| + \frac{1}{2\lambda} \int_{\R^{n}}|u-g|^{2}\,,
\end{equation}
where $|Du|$ is the total variation measure associated with the distributional gradient of $u$, and $\lambda$ is a positive parameter. We notice that the functional defined in \eqref{ROFproblem} is strictly convex, and it is not difficult to prove existence (and uniqueness!) of a solution. One could then be tempted to write the following  Euler-Lagrange equation associated with \eqref{ROFproblem}:  
\begin{equation}\label{ROFpde}
\lambda\mdiv \left(\frac{Du}{|Du|}\right) = u-g\,.
\end{equation}
However this is far from being correct, since one expects that some ``staircasing effects'' occurs in the solution, and therefore that its gradient vanishes on regions of positive Lebesgue measure. The correct way of writing the Euler-Lagrange equation can thus be found by means of convex analysis. We recall that the total variation of $Du$ is the convex functional defined by
\[
|Du|(\R^{n}) = \int_{\R^{n}}|Du| := \sup\left\{\int u \mdiv \xi:\ \xi\in C^{1}_{c}(\R^{n};\R^{n}),\ |g|\leq 1\right\}\,.
\]
We shall also set $J[u] = |Du|(\R^{n})$. Being $J[u]$ convex, we can consider its subdifferential at $u\in L^{2}(\R^{n})$:
\[
\de J[u] = \{v\in L^{2}(\R^{n}):\ J[u+w] \geq J[u] + \langle v,w\rangle\ \text{for all }w\in L^{2}(\R^{n})\}\,.
\] 
Then, the Euler-Lagrange relation derived from the minimality of $u$ with respect to problem \eqref{ROFproblem} is $0 \in \de J[u] + \frac{u-g}{\lambda}$ or, equivalently, 
\begin{equation}\label{ROFpdi}
\frac{g-u}{\lambda} \in \de J[u]\,.
\end{equation}
It is possible to show that the subdifferential $\de J[u]$ consists of the divergences of vector fields that ``calibrate'' the distributional gradient $Du$. More precisely, one can rewrite the Euler-Lagrange inclusion \eqref{ROFpdi} in the following, equivalent form: \textit{there exists a vector field $\xi_{u}\in L^{\infty}(\R^{n})$ such that $|\xi_{u}|\leq 1$, $\mdiv \xi_{u} \in L^{2}(\R^{n})$, $Du = \xi_{u} |Du|$ and 
\begin{equation}\label{divxi}
\mdiv \xi_{u} = \frac{u-g}{\lambda}\,.
\end{equation}
}
Let us assume from now on that $g = \chi_{\Om}$ is the characteristic function of some bounded Lipschitz domain $\Om$. The goal is to characterize the domains $\Om$ for which the solution $u$ of \eqref{ROFproblem} with $g = \chi_{\Om}$ is a ``scaled copy of $g$'', i.e. of the form $u = \mu \chi_{\Om}$, with $\mu\geq 0$. This means that the regularization produced by the ROF model \eqref{ROFproblem} determines, in this case, a change of the contrast, but not of the shape of the initial image $g = \chi_{\Om}$.

Following \cite{AltCasCha2005.1}, we say that a Lipschitz domain $\Om$ is \textit{calibrable} if $P(\Om)<\infty$ and if there exists a vector field $\xi\in L^{\infty}(\R^{n};\R^{n})$ such that $|\xi|\leq 1$, $\xi = \nu_{\Om}$ $\Hau^{n-1}$-almost everywhere on $\de\Om$, and 
\[
-\mdiv \xi = \frac{P(\Om)}{|\Om|} \chi_{\Om}
\]
in the distributional sense. This notion of calibrability is already present in the context of existence and uniqueness problems for graphs with prescribed mean curvature (see \cite{Giusti1978}). By using \eqref{divxi} one can derive the following result (see \cite{AltCasCha2005.1, AltCasCha2005}):
\begin{teo}The function $u_{\lambda} = \left(1 - \frac{P(\Om)}{|\Om|}\lambda\right)^{+}\chi_{\Om}$ is the unique minimizer of \eqref{ROFproblem} with $g = \chi_{\Om}$ if and only if $\Om$ is calibrable.
\end{teo}
\begin{proof}
First we consider the case $\lambda < \frac{|\Om|}{P(\Om)}$, so that
\[
u_{\lambda} = \left(1 - \frac{P(\Om)}{|\Om|}\lambda\right)\chi_{\Om}\,.
\]
Then we can easily check that $u_{\lambda}$ is the unique minimizer of \eqref{ROFproblem} if and only if there exists a vector field $\xi\in K$ satisfying $D\chi_{\Om} = \xi |D\chi_{\Om}|$ and such that \eqref{divxi} holds for $g = \chi_{\Om}$, which means that 
\[
-\mdiv \xi = \frac{P(\Om)}{|\Om|}\chi_{\Om}\,,
\]
that is, $\Om$ is calibrable. Concerning the case $\lambda \ge \frac{|\Om|}{P(\Om)}$, we observe that \eqref{divxi} is satisfied when $u=0$, $g = \chi_{\Om}$ and 
\[
\xi_{u} = \xi_{0} = \frac{|\Om|}{\lambda P(\Om)}\xi\,,
\]
where $\xi$ denotes a calibrating vector field for $\Om$. Note that $|\xi_{0}|\le 1$ in this case, thus $\mdiv \xi_{0}\in \de J[0]$.
\end{proof}
One can appreciate the close connection between ROF minimization and the Cheeger problem, through this notion of calibrability. To clarify this point, let us first recall the notion of \textit{mean-convexity}. We say that an open set $\Om\subset\R^{n}$ with finite perimeter is \text{mean-convex} if for any Borel set $F\subset\R^{n}$ such that $\Om\subset F$ we have $P(\Om)\leq P(F)$. In other words, $\Om$ minimizes the perimeter with respect to outer variations. Since the orthogonal projection onto a convex set is a $1$-Lipschitz map, by the area formula one can easily infer that (bounded) convex sets are also mean-convex (the converse being not true in general). 
The following proposition holds.
\begin{prop}\label{prop:calibrable}
Let $\Om$ be a Lipschitz domain. 
\begin{itemize}
\item[(i)] if $\Om$ is calibrable, then it is also mean-convex and self-Cheeger;

\item[(ii)] if $\Om$ is convex and self-Cheeger, then it is calibrable.
\end{itemize}
\end{prop}
The proof of claim (ii) of Proposition \ref{prop:calibrable} can be found in \cite{AltCasCha2005.1}. 
Here we only describe how to prove (i). Let $A\subset\Om$ be a relatively compact subdomain with smooth boundary, then by the divergence theorem applied to the calibrating vector field $\xi$ we get
\[
\frac{P(\Om)}{|\Om|} |A| = -\int_{A} \mdiv \xi = -\int_{\de A} \xi\cdot \nu_{A} \leq P(A)\,,
\]
whence $\frac{P(\Om)}{|\Om|} \leq \frac{P(A)}{|A|}$. Since we can fix a sequence of relatively compact subdomains $\Om_{h}$ converging to $\Om$ both in measure and in perimeter, we also conclude that $h(\Om) = \frac{P(\Om)}{|\Om|}$, that is, $\Om$ is self-Cheeger. To prove that $\Om$ minimizes the perimeter with respect to outer variations, we fix a bounded open set $F$ with Lipschitz boundary and strictly containing $\Om$, then we apply the divergence theorem to the calibrating vector field $\xi$ on $F\setminus \overline{\Om}$:
\begin{align*}
0 &= \int_{F\setminus \overline{\Om}} \mdiv \xi = \int_{\de F} \xi \cdot \nu_{F}\, d\Hau^{n-1} - \int_{\de \Om} \xi \cdot \nu_{F}\, d\Hau^{n-1} 
\\ &\leq P(F) - P(\Om)\,,
\end{align*}
which implies that $\Om$ is mean-convex.

\section{Some general results on the Cheeger problem}
\label{section:tre}
After recalling some basic facts about sets of finite perimeter, we shall present some general (and mostly known) results on the Cheeger problem for domains in $\R^{n}$. We shall also recall some more specific results valid for planar domains, that will be needed in Section \ref{section:quattro}.

For a given $x\in \R^{n}$ and $r>0$, we set $B_{r}(x)=\{y\in \R^{n}:\ |y-x|<r\}$, where $|v|$ is the Euclidean norm of the vector $v\in \R^{n}$. Given $A\subset \R^{n}$ we denote by $\chi_{A}$ its characteristic function. With a slight abuse of notation, we write $|A|$ for the Lebesgue (outer) measure of $A$. We then set $\om_{n} = |B_{1}(0)|$. We define the \textit{perimeter} of a Borel set $E$ as
\[
P(E) = \sup\left\{\int_{E}\mdiv g:\ g\in C^{1}_{c}(\R^{n};\R^{n}),\ |g|\leq 1\right\}\,.
\]
When $P(E)<+\infty$, we say that $E$ has finite perimeter (in $\R^{n}$). In this case, $P(E)$ coincides with the total variation of the distributional gradient of the characteristic function of $E$:
\[
P(E) = |D\chi_{E}|(\R^{n})\,,
\]
which more generally allows us to define the \textit{relative perimeter} 
\[
P(E;A) := |D\chi_{E}|(A)
\]
for any Borel set $A\subset \R^{n}$.  By Radon-Nikodym Theorem we can find a Borel $\R^{n}$-valued function $\nu_{E}$ such that $|\nu_{E}| = 1$ $|D\chi_{E}|$-almost everywhere and
\[
D\chi_{E} = -\nu_{E} |D\chi_{E}|\,.
\]
One can interpret $\nu_{E}$ as a \textit{generalized exterior normal} to the boundary of $E$. In order to clarify this concept, we recall the definition of \textit{reduced boundary} $\de^{*}E$. We say that $x\in \de^{*}E$ if $0<|E\cap B(x,r)| < \om_{n}r^{n}$ for all $r>0$ and
\[
\exists\,\nu_{E}(x) := -\lim_{r\to 0^{+}}\frac{D\chi_{E}(B_{r}(x))}{|D\chi_{E}|(B_{r}(x))}\,,\qquad |\nu_{E}(x)| = 1\,.
\]
Then we quote a classical result by De Giorgi \cite{DeGiorgi2006}:
\begin{teo}[De Giorgi]\label{teo:degiorgi}
Let $E$ be a set of finite perimeter, then 
\begin{itemize}
\item[(i)] $\de^{*}E$ is countably $\Hau^{n-1}$-rectifiable in the sense of Federer \cite{FedererBOOK};

\item[(ii)] for any $x\in \de^{*}E$, $\chi_{t(E-x)} \to \chi_{H_{\nu_{E}(x)}}$ in $L^{1}_{loc}(\R^{n})$ as $t\to +\infty$, where $H_{\nu}$ denotes the half-space through $0$ whose exterior normal is $\nu$;

\item[(iii)] for any Borel set $A$, $P(E;A) = \Hau^{n-1}(A\cap \de^{*}E)$;

\item[(iv)] $\int_{E}\mdiv g = \int_{\de^{*}E} g\cdot \nu_{E}\, d\Hau^{n-1}$ for any $g\in C^{1}_{c}(\R^{n};\R^{n})$.
\end{itemize}
\end{teo}
The perimeter functional extends the usual notion of $(n-1)$-dimensional measure of the boundary of a set, in the sense that, for instance, $P(E) = \Hau^{n-1}(\de E)$ for any bounded set $E$ with Lipschitz boundary. 
The advantage of using the perimeter functional instead of the Hausdorff measure in geometric variational problems is mainly due to the lower-semicontinuity and compactness properties stated in the following proposition (see, e.g., \cite{AmbFusPal2000}):
\begin{prop}[Lower-semicontinuity and compactness]\label{prop:semicomp}
Let $\Om\subset \R^{n}$ be an open set and let $(E_{j})_{j}$ be a sequence of Borel sets. We have the following well-known properties:
\begin{itemize}
\item[(i)] if $E$ is a Borel set, such that $\chi_{E_{j}}\to\chi_{E}$ in $L^{1}_{loc}(\Om)$, then $P(E;\Om) \leq \liminf\limits_{j} P(E_{j};\Om)$;
\item[(ii)] if there exists a constant $C>0$ such that $P(E_{j};\Om)\leq C$ for all $j$, then there exists a subsequence $E_{j_{k}}$ and a Borel set $E$ such that $\chi_{E_{j_{k}}}\to \chi_{E}$ in $L^{1}_{loc}(\Om)$.
\end{itemize}
\end{prop}
Other useful properties of the perimeter (invariance by isometries and scaling property, isoperimetric inequality, lattice property) are collected in the next proposition.
\begin{prop}
Given two Borel sets $E,F\subset \R^{n}$ of finite perimeter, $\lambda>0$ and an isometry $T:\R^{n}\to \R^{n}$, we have
\begin{align}
\label{Pomogeneo}
P(\lambda T(E)) &= \lambda^{n-1}P(E)\,,\\[3pt] 
\label{isopRn}
P(E) &\geq n\om_{n}^{\frac 1n}|E|^{\frac{n-1}{n}}\,,\\[3pt]
\label{reticolo}
P(E\cup F) + P(E\cap F) &\leq P(E) + P(F)\,.
\end{align}
\end{prop}

We point out that sets of finite perimeter can be extremely weird. For instance, 
let $G$ be the countable union of open balls of radius $2^{-i}$ centered at $q_{i},\ i\in \N$, where $(q_{i})_{i\in \N}$ is any enumeration of all points with rational coordinates in $\R^{2}$. By \eqref{reticolo} and lower-semicontinuity of the perimeter, $G$ has finite perimeter. However, its topological boundary has a positive Lebesgue measure (thus in particular its $(n-1)$-dimensional Hausdorff measure is $+\infty$). While generic sets of finite perimeter may thus be very irregular, a regularity theory is available in particular for \textit{minimizers} of the perimeter subject to a volume constraint (see \cite{Tamanini1982}). 
\begin{teo}[Regularity of perimeter minimizers with volume constraint]\label{teo:regolarita}
Let $\Om$ be a fixed open domain, and assume $E$ is a Borel set satisfying the following property: $P(E;\Om)<+\infty$ and for all Borel $F$ such that $E\difsim F \subset\subset \Om$ and $|F\cap \Om| = |E\cap \Om|$, it holds
\[
P(F;\Om) \leq P(E;\Om)\,.
\]
Then, $\de^{*}E \cap \Om$ is an analytic surface with constant mean curvature, and the singular set $(\de E \setminus \de^{*}E) \cap \Om$ is a closed set with Hausdorff dimension at most $n-8$. 
\end{teo}

We now focus on the Cheeger problem, and in doing so we first present some general properties of the Cheeger constant $h(\Om)$ and of Cheeger sets inside $\Om$, valid for any dimension $n\geq 2$ (see \cite{StrZie1997, KawFri2003,KawLac2006}).

\begin{prop}\label{prop:CheegerGenProp}
Let $\Om,\widetilde \Om\subset \R^{n}$ be bounded, open sets. Then the following properties hold.
\begin{itemize}
\item[(i)] If $\Om\subset \widetilde\Om$ then $h(\Om) \geq h(\widetilde\Om)$.

\item[(ii)] For any $\lambda>0$ and any isometry $T:\R^{n}\to\R^{n}$, one has $h(\lambda T(\Om)) = \frac{1}{\lambda} h(\Om)$.

\item[(iii)] There exists a (possibly non-unique) Cheeger set $E\subset \Om$, i.e. such that $\frac{P(E)}{|E|} = h(\Om)$.

\item[(iv)] If $E$ is Cheeger in $\Om$, then $E$ minimizes the relative perimeter among sets with the same volume; consequently, $\de E\cap \Om$ has the regularity stated in Theorem \ref{teo:regolarita}, and in particular $\de^{*}E\cap \Om$ is a hypersurface of constant mean curvature equal to $\frac{h(\Om)}{n-1}$.

\item[(v)] If $E$ is Cheeger in $\Om$ then $|E| \geq \om_{n}\left(\frac{n}{h(\Om)}\right)^{n}$.

\item[(vi)] If $E$ and $F$ are Cheeger in $\Om$, then $E\cup F$ and $E\cap F$ (if it is not empty) are also Cheeger in $\Om$.

\item[(vii)] If $E$ is Cheeger in $\Om$ and $\Om$ has finite perimeter, then $\de E\cap \Om$ can meet $\de^{*}\Om$ only in a tangential way, that is, for any $x\in \de^{*}\Om\cap \de E$ one has that $x\in \de^{*}E$ and $\nu_{E}(x) = \nu_{\Om}(x)$.

\end{itemize}
\end{prop}
\begin{proof}[Sketch of proof]
(i) and (ii) are immediate consequences of the definition of Cheeger constant and of \eqref{Pomogeneo} coupled with $|\lambda \Om| = \lambda^{n}|\Om|$. The proofs of (iii) and (iv) are accomplished by, respectively, Proposition \ref{prop:semicomp} and Theorem \ref{teo:regolarita}. The proof of (v) follows from the isoperimetric inequality \eqref{isopRn} and the fact that $P(E) = h(\Om)|E|$. To prove (vi) we apply \eqref{reticolo} and get
\begin{align*}
h(\Om) (|E\cup F| + |E\cap F|) &= h(\Om)(|E|+|F|)\\
&= P(E)+P(F)\\
&\geq P(E\cup F) + P(E\cap F)\\ 
&\geq h(\Om)(|E\cup F| + |E\cap F|)\,,
\end{align*}
hence all previous inequalities are actually equalities and this happens if and only if
\[
P(E\cup F) = h(\Om)|E\cup F|\qquad\text{and}\qquad P(E\cap F) = h(\Om)|E\cap F|\,,
\]
which proves (vi). While the proofs of (i)-(vi) are essentially known and can be found in the previously cited references, for the proof of (vii) we refer to \cite{LeoPra2014}.

\end{proof}

\begin{oss}\rm 
We notice that, by Proposition \ref{prop:CheegerGenProp} (v) and (vi), we can always find minimal Cheeger sets in $\Om$ (possibly not unique) and a unique maximal Cheeger set (this last can be obtained as the union of all minimal Cheeger sets of $\Om$). An example of a domain with two disjoint minimal Cheeger sets is shown in Figure \ref{fig:duecheeger}.
\end{oss}

We consider the problem of continuity of the Cheeger constant $h(\Om)$ with respect to some suitable notions of convergence of domains. In \cite{LeoPra2014} we prove Theorem \ref{teo:hcontinua} below (see also \cite{Parini_doktorarbeit2009} for the special case of convex domains). Since the proof is particularly simple, we quote it below with full details. 
\begin{teo}[Continuity properties of the Cheeger constant, \cite{LeoPra2014}]\label{teo:hcontinua}
Let $\Om,\Om_{j}\subset \R^{n}$ be nonempty open bounded sets for all $j\in \N$. If $\chi_{\Om_{j}} \to \chi_{\Om}$ in $L^{1}$, then 
\begin{equation}\label{hlsc}
\liminf_{j\to\infty} h(\Om_{j}) \geq h(\Om)\,.
\end{equation}
If in addition $\Om, \Om_{j}$ are sets of finite perimeter and $P(\Om_{j})\to P(\Om)$ as $j\to\infty$, then
\begin{equation}\label{hc}
\lim_{j\to\infty} h(\Om_{j}) = h(\Om)\,.
\end{equation}
\end{teo}
\begin{proof}
Let $E_{j}$ be a Cheeger set in $\Om_{j}$ (whose existence is guaranteed by Proposition \ref{prop:CheegerGenProp} (iii)). Without loss of generality we assume that $\liminf\limits_{j\to\infty} P(E_{j})$ is finite, then by Proposition \ref{prop:semicomp} we deduce that $\chi_{E_{j}}\to \chi_{E}$ in $L^{1}$ as $j\to\infty$, up to subsequences and for some Borel set $E$ with positive volume. Since $E_{j}\subset \Om_{j}$ and $\chi_{\Om_{j}}\to \chi_{\Om}$ in $L^{1}$ as $j\to\infty$, one immediately infers that $E\subset \Om$ up to null sets. Then by Proposition \ref{prop:semicomp} and by the convergence of $|E_{j}|$ to $|E|$, one has
\[
h(\Om) \leq \frac{P(E)}{|E|} \leq \liminf_{j\to\infty}\frac{P(E_{j})}{|E_{j}|}\,,
\] 
which proves \eqref{hlsc}. If in addition $P(\Om_{j})\to P(\Om)$ as $j\to\infty$, then we consider $E$ Cheeger in $\Om$ and define $E_{j} = \Om_{j}\cap E$. One can easily check that $E_{j} \to E$ and $E\cup \Om_{j}\to \Om$ in $L^{1}$, as $j\to\infty$. Therefore by \eqref{reticolo} we find
\begin{align*}
\limsup_{j\to\infty} P(E_{j}) &\leq P(E) + \limsup_{j\to\infty} P(\Om_{j}) - \liminf_{j\to\infty} P(E\cup \Om_{j})\\ 
&\leq P(E) + P(\Om) - P(\Om)\\ 
&= P(E)\,,
\end{align*}
which combined with \eqref{hlsc} gives \eqref{hc}.
\end{proof}

\subsection{The Cheeger problem in convex domains} 
Further properties of the Cheeger constant and of Cheeger sets are known when the domain $\Om$ is convex. In particular, we refer to \cite{AltCas2009} and to the references therein for the proof of the following result.
\begin{teo}\label{teo:convC11unique}
Let $\Om\subset\R^{n}$ be a convex domain. Then there exists a unique Cheeger set $E$ in $\Om$. Moreover, $E$ is convex and of class $C^{1,1}$.
\end{teo}
We remark that Theorem \ref{teo:convC11unique} was proved in \cite{CasChaNov2007} under stronger assumptions on $\Om$. The proof is essentially based on exploiting the link between the Cheeger problem and the capillary problem with zero gravity (i.e., with vertical contact at the boundary, see the discussion in the previous section). In particular, one has that a convex domain $\Om$ is self-Cheeger if and only if it is calibrable, and this happens precisely when $\Om$ is of class $C^{1,1}$ and the mean curvature of $\de\Om$ is bounded from above by $\frac{P(\Om)}{(n-1)|\Om|}$. 

More can be said about Cheeger sets of convex domains of the plane. For the proof of the following result, see \cite{StrZie1997, KawLac2006}.
\begin{teo}\label{teo:kawlac}
Let $\Om$ be a bounded convex set in $\R^{2}$. Then the unique Cheeger set $E$ of $\Om$ is the union of all balls of radius $r = h(\Om)^{-1}$ that are contained in $\Om$. Moreover, if we define the inner Cheeger set as
\begin{equation}\label{innercheeger}
E_{r} = \{x\in \Om:\ \dist(x,\de \Om)>r\}
\end{equation}
we have $E = E_{r}+B(0,r)$ (as a Minkowski sum) and it holds
\begin{equation}\label{KLformula}
|E_{r}| = \pi r^{2}\,.
\end{equation}
\end{teo}
The proof of Theorem \ref{teo:kawlac} is essentially based on \textit{Steiner's formulae} for area and perimeter of tubular neighbourhoods of convex sets in the plane (\cite{Steiner1840}): if $A\subset \R^{2}$ is a bounded convex set and $\rho>0$, then setting $A_{\rho} = A + B_{\rho}$ we have
\begin{align}
\label{Asteiner}
|A_{\rho}| &= |A| + \rho\, P(A) + \pi \rho^{2},\\ 
\label{Psteiner}
P(A_{\rho}) &= P(A) + 2\pi \rho\,.
\end{align}
 We recall that Steiner's formula \eqref{Asteiner} has been generalized by Weyl to $n$ dimensional domains with boundary of class $C^{2}$ (the so-called tube formula, see \cite{Weyl1939}) and then by Federer \cite{Federer1959} under the assumption of \textit{positive reach}, that we introduce hereafter. Given $K\subset \R^{n}$ compact, we define the \textit{reach} of $K$ as
\begin{align*}
 \reach(K) = \sup\{\e\ge 0:\ &\dist(x,K)\le \e\ \impl\ \text{$x$ has a unique}\\ & \text{projection onto $K$}\}\,.
\end{align*}
We say that $K$ has positive reach if $\reach(K)>0$. Notice that if $K$ is convex, then $\reach(K) = +\infty$. It is convenient to introduce the \textit{outer Minkowski content} of an open bounded set $A$, defined as
\[
\mathcal M(A) = \lim_{\rho\to 0}\frac{|A_{\rho}|-|A|}{\rho}\,,
\]
provided that the limit exists. 
Then we have the following result (see \cite{LeoPra2014}).
\begin{prop}
Let $A\subset \R^{2}$ be a bounded open set with Lipschitz boundary. Let us assume that $\reach(\overline{A})>0$. Then $P(A)<+\infty$ and Steiner's formulae \eqref{Asteiner}, \eqref{Psteiner} hold for all $0<\rho<\reach(\overline{A})$.  
\end{prop}

\begin{oss}\rm
To see how the inner Cheeger formula \eqref{KLformula} can be used to derive information on the Cheeger problem for convex planar domains, we compute the Cheeger constant of a unit square $Q = (0,1)^{2}$. First, we observe that the inner Cheeger set of $Q$ is a concentric square of side length $1-2r$. Therefore \eqref{KLformula} becomes
\[
(1-2r)^{2} = \pi r^{2}\,,
\]
and by coupling this equation with the condition $1-2r>0$ we infer after some elementary computations that 
\[
r = \frac{1}{2+\sqrt{\pi}}\,,
\]
whence $h(Q) = 2+\sqrt{\pi}$. A general algorithm for computing the Cheeger constant of a convex polygon (with some extra property, i.e., that there is a one-to-one correspondence between the connected components of the boundary of the Cheeger set in the interior of the polygon and the vertices of the polygon) can be found in \cite{KawLac2006}.
\end{oss}

\subsection{Some further results about Cheeger sets in $\R^{2}$}
Let $E$ be a Cheeger set inside an open bounded domain $\Om\subset \R^{2}$, and set $r = h(\Om)^{-1}$ as before. Then a first, general fact is that a connected component of $\de E\cap \Om$ is an arc of radius $r$, that cannot be longer than $\pi r$ (i.e., it can be at most a half-circle).
\begin{lemma}[\cite{LeoPra2014}]\label{lemma:180}
Let $\de E\cap \Om$ be nonempty and let $S$ be one of its connected components. Then $S$ is an arc of circle of radius $r$, whose length does not exceed $\pi r$.
\end{lemma}

An apparently, very intuitive property of a planar Cheeger set $E$ could be the fact that $E$ satisfies an internal ball condition of radius $r = \frac{|E|}{P(E)}$ (i.e., that it is a union of balls of radius $r$). However, this property is false in general (see Figure \ref{fig:cheegerinunion} and, in particular, Example \ref{bow-tie}). Anyway, the following result holds true: as soon as a maximal Cheeger set $E$ in $\Om$ contains some ball $B_{r}(x_{0})$, one can roll this ball inside $\Om$ following any sufficiently smooth path of centers, and in doing so the moving ball will remain inside $E$.but without exiting from $E$.
\begin{teo}[Moving ball, \cite{LeoPra2014}]\label{teo:movingball}
Let $r=1/h(\Om)$ and let $E$ be a maximal Cheeger set in $\Om$ containing a ball $B_{r}(x_{0})$. Assume that there exists a curve $\gamma:[0,1]\to \Om$ of class $C^{1,1}$ and curvature bounded by $h(\Om)$, such that $x_{0} = \gamma(0)$ and $B_{r}(\gamma(t))\subset \Om$ for all $t\in [0,1]$. Then $B_{r}(\gamma(t))\subset E$ for all $t\in [0,1]$.
\end{teo}

\begin{oss}\label{oss:movingball}\rm
The requirement in Theorem \ref{teo:movingball} of maximality of $E$ can be dropped whenever the moving ball remains at a positive distance from $\de\Om$. In this case, one can prove by using Lemma \ref{lemma:180} that the moving ball will never intersect $\de E$. 
\end{oss}

\section{Characterization of Cheeger sets in planar strips}
\label{section:quattro} 
In \cite{KrePra2011}, D.~Krej\v{c}i\v{r}\'\i{k} and A.~Pratelli consider the Cheeger problem for a class of generically non-convex planar domains, called \textit{strips}. Let $\gamma:[0,L]\to \R^{2}$ be a curve of class $C^{1,1}$ parametrized by arc-length, such that the modulus of its curvature is bounded by $1$. For $t\in [0,L]$ we denote by $\sigma(t)$ the relatively open segment of length $2$ whose midpoint is $\gamma(t)$ and such that $\gp(t)$ is orthogonal to $\sigma(t)$. We also assume that $0\leq t_{1}<t_{2}\leq L$ implies $\sigma(t_{1}) \cap \sigma(t_{2}) = \emptyset$ (no-crossing condition). Then, the set 
\[
\strip = \Int\left(\bigcup_{t\in [0,L]} \sigma(t)\right)
\]
is an \textit{open strip} of width $2$ and length $L$ (here $\Int(A)$ denotes the set of interior points of $A$). We call $\gamma$ the \textit{spinal curve} of the strip $\strip$. If the no-crossing condition holds for all $t_{1}<t_{2}\in (0,L)$, but $\sigma(0) = \sigma(L)$, then we say that $\strip$ is a \textit{closed strip} (of course, this requires that the curve $\gamma$ is closed, too). In particular, an open strip is a $C^{1,1}$-diffeomorphic image of $(0,L)\times (-1,1)$, while a closed strip is a $C^{1,1}$-diffeomorphic image of $[0,L]\times (-1,1)$ with identification of points $(0,y)$ and $(L,y)$. More precisely we can take $(t,u)\in [0,L]\times (-1,1)$ and define the map
\[
\Psi(t,u) = \g(t) + u\ \nu(t)\,,
\]
where $\nu(t)$ denotes the counter-clockwise rotation of the unit vector $\gp(t)$ by $90$ degrees. In the following we shall focus on open strips, as the Cheeger problem for closed ones has been completely treated in \cite{KrePra2011}.  One can check that the map $\Psi$ defined above is a diffeomorphism of class $C^{1,1}$ between the rectangle $(0,L)\times (-1,1)$ and the (open) strip $\strip$. Using the $(t,u)$ coordinate system, i.e. the representation of a generic point $x$ of the strip by means of its coordinates $(t,u) = \Psi^{-1}(x)$, can sometimes be of help. 

\begin{figure}[ht]
  \centering
  \includegraphics[scale=0.9]{metap_all-0}
  \caption{A planar strip $\strip$.}
  \label{fig:strip}
\end{figure}

Up to a scaling, we can more generally define strips of width $2s$ (in this case we must require that the modulus of the curvature of $\gamma$ is smaller than $1/s$). Without loss of generality, we shall consider only strips of width $s=2$. Moreover, we shall also assume that the curvature of $\gamma$ is everywhere $<1$, as we can recover the case $\leq 1$ by approximation, owing to Theorem \ref{teo:hcontinua}. Strips naturally appear in spectral problems, as they model $2$-dimensional waveguides (see \cite{DucExn1995,KreKri2005}). In this sense, a number of interesting questions involve the spectral behaviour of a strip when its length $L$ becomes very large. In \cite{KrePra2011} the authors prove some results specifically on the Cheeger problem for strips. First, they show that closed strips are Cheeger in themselves. Then they prove by means of a suitable symmetrization technique the following bounds on the Cheeger constant of a strip (see Theorem 10 in \cite{KrePra2011}).
\begin{teo}[\cite{KrePra2011}]\label{teo:KrePra}
Let $\strip$ be a strip of length $L$ and width $2$. Then
\begin{equation}\label{stimaord1}
1 + \frac{1}{400\,L} \leq h(\strip) \leq 1 + \frac 2L\,.
\end{equation}
\end{teo}
In \cite{LeoPra2014} we push forward the analysis done in \cite{KrePra2011} and, by means of a finer characterization of Cheeger sets inside open strips, we prove the following result.
\begin{teo}[\cite{LeoPra2014}]\label{teo:asintotico}
Let $\strip$ be an open strip of length $L\geq   \frac{9\pi}{2}$ and width $2$. Then 
\begin{align}\label{stripasymptotic}
h(\strip) = \left(1+ \frac{\pi}{2L} + O(L^{-2})\right)\qquad \text{as }L\to +\infty\,.
\end{align}
\end{teo}
The asymptotic estimate \eqref{stripasymptotic} is optimal. Its derivation is based on a key result proved in \cite{LeoPra2014}. This result (Theorem \ref{teo:unionedipalle2} recalled below) essentially shows that, concerning the Cheeger problem, strips are not too different from convex domains. 
\begin{teo}[\cite{LeoPra2014}]\label{teo:unionedipalle2}
Let $\strip$ be an open strip of length $L\geq 9\pi/2$, and let $r=h(\strip)^{-1}$. Assume $E$ is a Cheeger set of $\strip$. Then there exists two continuous functions $\rho^{+},\rho^{-}:[0,L]\to [-1,1]$ such that 
\begin{equation}\label{intergrafico}
E = \Psi\left(\{(t,s):\ 0< t< L,\ \rho^{-}(t)<s<\rho^{+}(t)\}\right)\,.
\end{equation}
Moreover, $E$ is unique and coincides with the union of all balls of radius $r$ contained in $\strip$, it is simply connected and can be obtained as the Minkowski sum $E = E_{r}+B_{r}$, where
\[
E_{r} = \{x\in \strip:\ \dist(x,\de\strip)\geq r\}
\]
is a set with Lipschitz boundary and positive reach $\reach(E_{r})\geq r$. Finally, the inner Cheeger formula
\begin{equation}\label{ICformula}
|E_{r}| = \pi r^{2}
\end{equation}
holds true.
\end{teo}

\begin{oss}\rm
We stress that the conclusions of Theorem \ref{teo:unionedipalle2} (in particular, the fact that the Cheeger set $E$ is the union of balls of radius $r$ contained in $\strip$, and that \eqref{ICformula} holds true) are not satisfied by any planar domain. Two examples showing that no inclusion between Cheeger sets and unions of balls of radius $r$ is generally true, are given in the last section (Examples \ref{ex:cheegerinunion} and \ref{bow-tie}). Concerning the inner Cheeger formula \eqref{ICformula}, there exists a star-shaped domain whose Cheeger set is the union of all included balls of radius $r$, but for which the formula fails (see Example \ref{loose bow-tie}). 
\end{oss}

Theorem \ref{teo:asintotico} directly follows from Theorem \ref{teo:unionedipalle2}. Indeed, by the special geometric properties of $\strip$ we infer that
\[
2(L-9\pi) (1-r) \leq |E_{r}| \leq 2L(1-r)\,.
\]
By combining these two inequalities with the inner Cheeger formula \eqref{ICformula}, we finally get
\[
2(L-9\pi) (1-r) \leq \pi r^{2} \leq 2L(1-r)\,,
\]
which implies \eqref{stripasymptotic} by an elementary computation.

We synthetically present the main ideas and tools, which the proof of Theorem \ref{teo:unionedipalle2} is based on. Again, we refer the reader to \cite{LeoPra2014} for the details. 
We start recalling two key lemmas that are used in the proof of Theorem \ref{teo:unionedipalle2}. The first lemma states that, if $E$ is a Cheeger set in $\strip$, and the length of $\strip$ is large enough, then any osculating ball to $\de E\cap \strip$ is entirely contained in $\strip$ (see Figure \ref{fig:arcball}).
\begin{figure}[ht]
  \centering
  \includegraphics[scale=1]{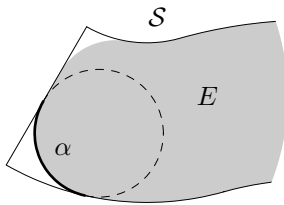}
  \caption{The arc-ball property.}
  \label{fig:arcball}
\end{figure}

\begin{lemma}[Arc-ball property]\label{lemma1.1}
Let $E$ be a Cheeger set inside a strip $\strip$ of length $L\geq \frac{9\pi}2$. Set $r = h(\strip)^{-1}$. Then $\de E\cap \strip$ is non-empty, and for any circular arc $\alpha$ contained in $\de E\cap \strip$ the unique ball $B_{r}$, such that $\alpha\subset \de B_{r}$, is contained in $\strip$.
\end{lemma}

The second lemma establishes a \textit{ball-to-ball} connectivity property of a generic strip, that is, the possibility of connecting two balls of radius $r$ that are contained in $\strip$ by moving one of them towards the other, following a suitable path of centers with controlled curvature and preserving the inclusion in $\strip$ (see Figure \ref{fig:ball2ball}).
\begin{figure}[ht]
  \centering
  \includegraphics[scale=1]{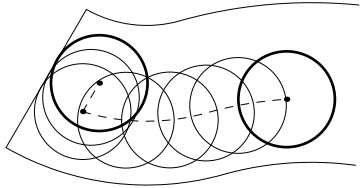}
  \caption{The ball-to-ball property.}
  \label{fig:ball2ball}
\end{figure}

\begin{lemma}[Ball-to-ball property]\label{lemma1.2}
If $B_r(x_{0})$ and $B_r(x_{1})$ are two balls of radius $r\leq 1$, and both are contained in a strip $\strip$, then there exists a piece-wise $C^{1,1}$ curve $\beta:[0,1]\to \strip$ such that $\beta(0) = x_{0}$, $\beta(1) = x_{1}$, the curvature of $\beta$ is smaller than $r^{-1}$, and $B_{r}(\beta(t))\subset \strip$ for all $t\in (0,1)$.
\end{lemma}
The main difficulties in proving Lemmas \ref{lemma1.1} and \ref{lemma1.2} are of topological type. Roughly speaking, one has to exploit the structural properties of the strip in order to exclude some weird behaviour of its boundary. As it happens for many intuitively clear statements concerning planar objects, proving such lemmas is not as easy as one could imagine at a first sight. For instance, we found no particular simplifications in those proofs by working in the $(t,u)$ coordinate system: this can be understood if one considers that the pre-image of a ball with respect to the map $\Psi$ is no more a ball in the $(t,u)$ coordinates. In several steps of the proofs we find it convenient to argue by contradiction, since a number of (a-posteriori impossible) situations, like for instance the one where an internal ball of radius $\e<r$ is tangent to more than one point of $\de^{+}\strip$, or the other where two distinct balls of radius $r$ centered on the spinal curve $\g$ are both tangent to $\sigma(0)$, must be excluded.

With these two lemmas at hand, we can prove Theorem \ref{teo:unionedipalle2}. Hereafter we provide only a sketch of its proof.
\begin{proof}[Proof sketch of Theorem \ref{teo:unionedipalle2}]
First of all, we show that there exist exactly four balls of radius $r$ and centers $x_{0,0},x_{1,0},x_{0,1},x_{1,1}$, such that the boundary of $B_{r}(x_{i,j})$ contains the connected component of $\de E\cap \strip$ that is tangent to
\begin{itemize}
\item[] $\sigma(0)$ and $\de^{-}\strip$ if $i=j=0$;
\item[] $\sigma(0)$ and $\de^{+}\strip$ if $i=0$ and $j=1$;
\item[] $\sigma(L)$ and $\de^{-}\strip$ if $i=1$ and $j=0$;
\item[] $\sigma(L)$ and $\de^{+}\strip$ if $i=j=1$.
\end{itemize}
This can be accomplished by combining Lemma \ref{lemma1.1}, Lemma \ref{lemma1.2}, and Theorem \ref{teo:movingball}. With this result in force, we are able to define the two functions 
$\rho^{\pm}$ satisfying \eqref{intergrafico}, which completely characterize the boundary of $E$. The simply connectedness of $E$ is immediate, as its homeomorphic representation in coordinates $(t,u)$ clearly satisfies this property. On the other hand, we can prove that the union $U_{r}$ of all balls of radius $r$ that are contained in $\strip$ admits in the $(t,u)$ coordinate system the same representation as $E$:
\[
U_{r} = \Psi\left(\{(t,u):\ 0< t< L,\ \rho^{-}(t)<u<\rho^{+}(t)\}\right)\,,
\] 
hence $E=U_{r}$. Finally, the properties concerning the inner Cheeger set $E_{r}$ can be proved as follows. For $i,j=0,1$ we denote by $a_{i,j}$ the first coordinate of $x_{i,j}$ in the $(t,u)$ representation, then set
\begin{itemize}
\item[] $p_{i,j}=$ the orthogonal projection of $x_{i,j}$ onto $\sigma(iL)$; 
\item[] $Q_{i}=$ the rectangle of vertices $x_{i,1},x_{i,0},p_{i,1},p_{i,0}$;
\item[] $D_{i,j}=$ the circular sector with center $x_{i,j}$ and boundary arc $S_{i,j}$;
\item[] $R_{+}=$ the region spanned by $\sigma(t;(1-r,1))$ as $t\in [a_{0,1},a_{1,1}]$;
\item[] $R_{-}=$ the region spanned by $\sigma(t;(-1,r-1))$ as $t\in [a_{0,0},a_{1,0}]$.
\end{itemize}
In the above definitions, $\sigma(t;A)$ denotes the set $\{\g(t) + u\,\nu(t):\ u\in A\}$. Consequently we have the decomposition
\[
E\setminus E_{r} = Q_{0}\cup Q_{1} \cup \bigcup_{i,j=0}^{1} D_{i,j}\cup R_{+}\cup R_{-}\,.
\]
Finally, we show that $E_{r}$ has a Lipschitz boundary and that any points of $E\setminus E_{r}$ has a unique projection onto $E_{r}$ (which can be more precisely identified according to the above decomposition, see Figure \ref{fig:stripdecomp}). We can thus apply Steiner's formulae \eqref{Asteiner} and \eqref{Psteiner}, as in the proof of Theorem \ref{teo:kawlac}, and obtain the inner Cheeger formula \eqref{ICformula}, thus concluding the proof of the theorem.
\end{proof}
\begin{figure}[ht]
  \centering
  \includegraphics[scale=0.9]{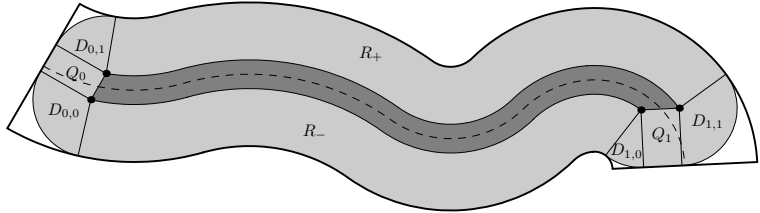}
  \caption{The decomposition of the Cheeger set. The inner Cheeger set $E_{r}$ is colored in dark grey. One can see the eight regions of the decomposition of $E\setminus E_{r}$ colored in light grey.}
  \label{fig:stripdecomp}
\end{figure}

\medskip

\section{Some planar examples}
We conclude by collecting some examples of non-convex planar domains, together with their Cheeger sets.

We start from an example of a domain $G$, whose Cheeger set is strictly contained in the union of balls of radius $r = h(G)^{-1}$ that are contained in $G$. 
\begin{example}[\cite{KawLac2006}]\label{ex:cheegerinunion}\rm 
Let $G$ be the union of two disjoint balls $B_{1}$ and $B_{\frac 23}$, of radii $1$ and $\frac 23$ respectively (see Figure \ref{fig:cheegerinunion}). One has $\frac{P(G)}{|G|} = \frac{30}{13} >2$. It is not difficult to check that the Cheeger set $E$ of $G$ coincides with $B_{1}$, hence $h(G) = 2$. However, $G$ coincides with the union of all balls of radius $r = h(G)^{-1} = \frac 12$ contained in $G$, which is therefore strictly larger than $E$. 
\begin{figure}[ht]
  \centering
  \includegraphics[scale=1]{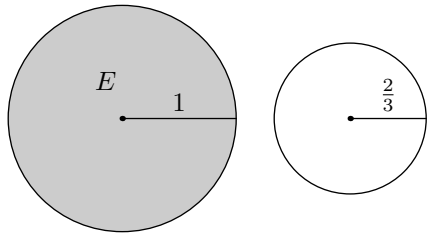}
  \caption{A union of two disjoint balls $B_{1}$ and $B_{\frac 23}$, whose Cheeger set $E$ coincides with the largest ball $B_{1}$.}
  \label{fig:cheegerinunion}
\end{figure}
\end{example}

The next example shows a Cheeger set $\bowtie$ \textit{strictly containing} the union of all balls of radius $h(\bowtie)^{-1}$ contained in $\bowtie$. This example and the one depicted in Figure \ref{fig:cheegerinunion} show that, in general, no inclusion holds between a Cheeger set of $\Om$ and the union of all balls of radius $r = h(\Om)^{-1}$ contained in $\Om$.
\begin{example}[\cite{LeoPra2014}]\label{bow-tie}\rm
Let us consider a unit-side equilateral triangle $T$, as in Figure \ref{fig:bowtie}, together with its Cheeger set $E_{T}$ (depicted in grey). Then, cut $T$ with the vertical line tangent to $E$ and reflect the portion on the left to the right, as shown in the picture. This produces a bow-tie $\bowtie$.
\begin{figure}[ht]
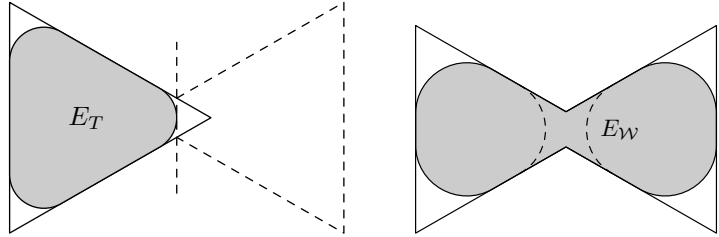

  \centering
  \includegraphics[scale=1]{metap_all-3}
  \qquad
  \includegraphics[scale=0.9]{metap_all-4}
  \caption{The construction of the bow-tie $\mathcal W$ (left) and the Cheeger set $E_{\mathcal W}$ in the bow-tie (right). Notice that the region between the two dashed lines in the picture on the right is the difference between the Cheeger set $E_{\mathcal W}$ and the (strictly smaller) union of all balls of radius $r$ included in $\mathcal W$.}
  \label{fig:bowtie}
\end{figure}
Let now $E_{\bowtie}$ be a Cheeger set inside $\bowtie$. By the $2$-symmetry of $\bowtie$ one can infer the $2$-symmetry of $E_{\bowtie}$. On the other hand, $E_{\bowtie}$ cannot have a connected component $F$ completely contained in $T$, since otherwise $F$ would be Cheeger inside $\bowtie$ and, at the same time, it would coincide with $E_{T}$. But then $E_{T}\cup E_{T}'$ (where we have denoted by $E_{T}'$ the reflected copy of $E_{T}$ with respect to the cutting line) would be Cheeger in $\bowtie$, which is not possible since $\de (E_{T}\cup E_{T}')\cap \bowtie$ is not everywhere smooth, as it should according to Proposition \ref{prop:CheegerGenProp}. Being necessarily $\de E_{\bowtie}\cap \bowtie$ equal to a finite union of circular arcs with the same curvature $=h(\bowtie)$, it is not difficult to rule out all possibilities except the one in which $\de E_{\bowtie} \cap \bowtie$ is composed by four congruent arcs, one for each convex corner in the boundary of $\bowtie$. Moreover one has the strict inequality $h(\bowtie)<h(T)$, therefore the union of all balls of radius $h(\bowtie)^{-1}$ contained in $\bowtie$ does not contain $E_{\bowtie}$ (indeed, some small region around the two concave corners cannot be covered by those balls).  
\end{example}

The next example is obtained as a slight variation of Example \ref{bow-tie}. In this case, the resulting Cheeger set is simply connected, while the inner Cheeger set is disconnected. As a result, we derive the impossibility for the inner Cheeger formula \eqref{ICformula} to hold.
\begin{example}[\cite{LeoPra2014}]\label{loose bow-tie}\rm
Take the bow-tie $\bowtie$ constructed in the previous example and vertically move the two concave corners a bit far apart. By the continuity of the Cheeger constant (see \eqref{hc}) we infer the existence of some minimal displacement of the two corners, such that the Cheeger set in the modified bow-tie $\widetilde\bowtie$ actually coincides with the union of all balls of radius $r = h(\widetilde \bowtie)^{-1}$. This corresponds to the situation represented in Figure \ref{fig:loosebowtie}.
\begin{figure}[ht]
  \centering
  \includegraphics[scale=1]{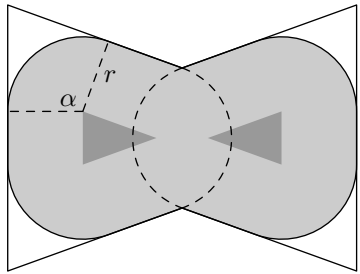}
  \caption{A loose bow-tie for which the inner Cheeger formula does not hold.}
  \label{fig:loosebowtie}
\end{figure}
It is then easy to check that the formula $|E_{r}| = \pi r^{2}$ does not hold in this case, essentially because the inner Cheeger set $E_{r}$ (depicted in dark grey) does not satisfy $\reach(E_{r}) \geq r$. We also notice that, while the Cheeger set $E$ is connected, the inner Cheeger set $E_{r}$ is disconnected. Finally, one can easily check that the true formula, that is satisfied by the inner Cheeger set in this case, is
\[
|E_{r}| = \big(2\alpha + \sin(2\alpha)\big) r^{2}>\pi r^{2}\,,
\]
where $\alpha$ is the angle depicted in Figure \ref{fig:loosebowtie}.
\end{example}

Before getting to the last examples, we recall a result of generic uniqueness for the Cheeger set inside a domain $\Om\subset \R^{n}$, proved in \cite{CasChaNov2010}:
\begin{teo}[\cite{CasChaNov2010}]\label{teo:genericunico}
Let $\Om\subset \R^{n}$ be any bounded open set, and let $\e>0$ be fixed. Then there exists an open set $\Om_{\e}\subset \Om$, such that $|\Om\setminus \Om_{\e}|<\e$ and the Cheeger set of $\Om_{\e}$ is unique.
\end{teo}
\begin{proof}[Idea of proof]
Let $E$ be a minimal Cheeger set of $\Om$, and let $\om_{\e}$ be a relatively compact, open subset of $\Om$ with smooth boundary, such that $|\Om\setminus \om_{\e}|<\e$. Define $\Om_{\e} = E \cup \om_{\e}$, then by an application of the strong maximum principle for constant mean curvature hypersurfaces one can show that $E$ is the unique Cheeger set of $\Om_{\e}$.
\end{proof}
\begin{example}[\cite{KawLac2006}]\label{ex:duecheeger}\rm 
Figure \ref{fig:duecheeger} shows a simply connected domain consisting of two congruent squares connected by a small strip. Each square with suitably rounded corners is a minimal Cheeger set, while their union is the maximal Cheeger set of the domain.  
\begin{figure}[ht]
  \centering
  \includegraphics[scale=1.4]{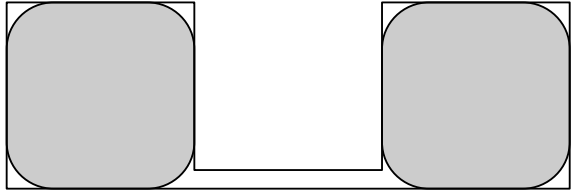}
  \caption{A simply connected domain whose Cheeger set is not unique.}
  \label{fig:duecheeger}
\end{figure}
\end{example}
A more sophisticated example of non-uniqueness is constructed below. Indeed one may ask whether it is possible to find a domain admitting infinitely many distinct Cheeger sets. The answer to this question is in the affirmative, as shown by the following example (we point out that a similar example was numerically discussed by E. Parini in his master degree thesis \cite{Parini_tesilaurea2006}).
\begin{example}[\cite{LeoPra2014}]\label{ex:pinocchio}
\rm 
Let $\pino_{\theta}$ be the union of a unit disc $B_{1}$ centered at $(0,0)$ and a disc of radius $r=\sin \theta$ and center $(\cos \theta,0)$, where $\theta\in(0,\pi/2)$ will be chosen later. The perimeter of $\pino_{\theta}$ is 
\[
P(\theta) = 2(\pi -\theta) + \pi \sin\theta, 
\]
while its area is
\[
A(\theta) = (\pi - \theta) + \sin\theta\, \cos\theta +\frac{\pi\sin^{2}\theta}{2}\,.
\]
\begin{figure}[ht]
  \centering
  \includegraphics[scale=1]{metap_all-9}
  \caption{The set $\pino(\theta)$.}
  \label{fig:pinocchio1}
\end{figure}
Then one shows the existence and uniqueness of $\theta_{0}\in (0,\pi/2)$ such that
\[
\frac{P(\theta_{0})}{A(\theta_{0})} = \frac{1}{\sin \theta_{0}}\,,
\]
that is,
\begin{equation}\label{thetavincolo}
2(\pi-\theta_{0})\sin\theta_{0} +\frac{\pi}{2}\sin^{2}\theta_{0} - (\pi - \theta_{0}) - \frac{\sin(2\theta_{0})}{2} = 0\,.
\end{equation}
Now we set for brevity $\pino_{0} = \pino_{\theta_{0}}$ and observe that the ratio $\frac{P(\theta_{0})}{A(\theta_{0})}$ equals the inverse of the radius of the smaller arc inside $\de\pino_{0}$. Then by a direct comparison with other possible competitors one infers that $\pino_{0}$ is Cheeger in itself. Now we consider the one parameter family $\pino_{t}$, $t\in [0,+\infty)$ of sets obtained by ``elongating the nose'' of $\pino_{0}$ (see Figure \ref{fig:pinocchio2}). It turns out that $\pino_{t}$ is Cheeger in $\pino_{\tau}$ whenever $t\leq \tau$, and this property is stable if one even ``bends the nose'' of $\pino_{t}$. Indeed, the Cheeger ratio of $\pino_{t}$ is constantly equal to $\frac{1}{\sin \theta_{0}}$.
\begin{figure}[ht]
  \centering
  \includegraphics[scale=1]{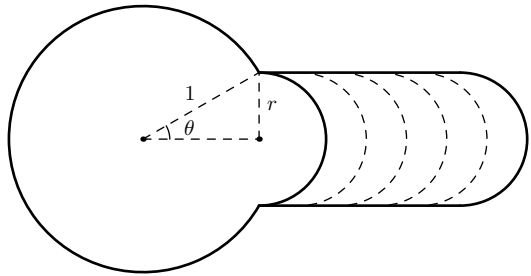}
  \caption{The one-parameter family of Cheeger sets.}
  \label{fig:pinocchio2}
\end{figure}
\end{example}
\medskip


\section*{Acknowledgements}
We thank the Department Mathematik -- Universit\"at Erlangen-N\"urnberg, and in particular Aldo Pratelli, for kind hospitality and financial support through ERC Starting Grant 2010 ``AnOptSetCon''. We also thank Carlo Nitsch and Antoine Henrot for making us aware of Enea Parini's work.


\begin{thebibliography}{10}

\bibitem{AltCas2009}
F.~Alter and V.~Caselles.
\newblock Uniqueness of the {C}heeger set of a convex body.
\newblock {\em Nonlinear Anal.}, 70(1):32--44, 2009.

\bibitem{AltCasCha2005.1}
F.~Alter, V.~Caselles, and A.~Chambolle.
\newblock {A characterization of convex calibrable sets in $\Bbb R^N$.}
\newblock {\em Math. Ann.}, 332(2):329--366, 2005.

\bibitem{AltCasCha2005}
F.~Alter, V.~Caselles, and A.~Chambolle.
\newblock {Evolution of characteristic functions of convex sets in the plane by
  the minimizing total variation flow.}
\newblock {\em Interfaces Free Bound.}, 7(1):29--53, 2005.

\bibitem{AmbFusPal2000}
L.~Ambrosio, N.~Fusco, and D.~Pallara.
\newblock {\em Functions of bounded variation and free discontinuity problems},
  volume 254.
\newblock Clarendon Press Oxford, 2000.

\bibitem{ButCarCom2007}
G.~Buttazzo, G.~Carlier, and M.~Comte.
\newblock On the selection of maximal {C}heeger sets.
\newblock {\em Differential Integral Equations}, 20(9):991--1004, 2007.

\bibitem{CarComPey2009}
G.~Carlier, M.~Comte, and G.~Peyr{\'e}.
\newblock Approximation of maximal {C}heeger sets by projection.
\newblock {\em M2AN Math. Model. Numer. Anal.}, 43(1):139--150, 2009.

\bibitem{CasChaMolNov2008}
V.~Caselles, A.~Chambolle, S.~Moll, and M.~Novaga.
\newblock A characterization of convex calibrable sets in {$\Bbb R^N$} with
  respect to anisotropic norms.
\newblock {\em Ann. Inst. H. Poincar\'e Anal. Non Lin\'eaire}, 25(4):803--832,
  2008.

\bibitem{CasChaNov2007}
V.~Caselles, A.~Chambolle, and M.~Novaga.
\newblock Uniqueness of the {C}heeger set of a convex body.
\newblock {\em Pacific J. Math.}, 232(1):77--90, 2007.

\bibitem{CasChaNov2010}
V.~Caselles, A.~Chambolle, and M.~Novaga.
\newblock {Some remarks on uniqueness and regularity of Cheeger sets.}
\newblock {\em Rend. Semin. Mat. Univ. Padova}, 123:191--201, 2010.

\bibitem{CasFacMei2009}
V.~Caselles, G.~Facciolo, and E.~Meinhardt.
\newblock {Anisotropic Cheeger sets and applications.}
\newblock {\em SIAM J. Imaging Sci.}, 2(4):1211--1254, 2009.

\bibitem{CasMirNov2010}
V.~Caselles, M.~j. Miranda, and M.~Novaga.
\newblock {Total variation and Cheeger sets in Gauss space.}
\newblock {\em J. Funct. Anal.}, 259(6):1491--1516, 2010.

\bibitem{ChaLio1997}
A.~Chambolle and P.-L. Lions.
\newblock Image recovery via total variation minimization and related problems.
\newblock {\em Numer. Math.}, 76(2):167--188, 1997.

\bibitem{Cheeger1970}
J.~Cheeger.
\newblock A lower bound for the smallest eigenvalue of the {L}aplacian.
\newblock {I}n problems in analysis (papers dedicated to {S}olomon {B}ochner,
  1969, 195-199). Princeton Univ. Press, 1970.

\bibitem{DeGiorgi2006}
E.~De~Giorgi.
\newblock {\em Selected papers}.
\newblock Springer-Verlag, Berlin, 2006.

\bibitem{DucExn1995}
P.~Duclos and P.~Exner.
\newblock Curvature-induced bound states in quantum waveguides in two and three
  dimensions.
\newblock {\em Rev. Math. Phys.}, 7:73--102, 1995.

\bibitem{Federer1959}
H.~Federer.
\newblock Curvature measures.
\newblock {\em Trans. Amer. Math. Soc.}, 93:418--491, 1959.

\bibitem{FedererBOOK}
H.~Federer.
\newblock {\em Geometric measure theory}, volume 153 of {\em Die Grundlehren
  der mathematischen Wissenschaften}.
\newblock Springer-Verlag New York Inc., New York, 1969.

\bibitem{Giaquinta1974}
M.~Giaquinta.
\newblock On the {D}irichlet problem for surfaces of prescribed mean curvature.
\newblock {\em Manuscripta Math.}, 12:73--86, 1974.

\bibitem{Giusti1978}
E.~Giusti.
\newblock On the equation of surfaces of prescribed mean curvature. {E}xistence
  and uniqueness without boundary conditions.
\newblock {\em Invent. Math.}, 46(2):111--137, 1978.

\bibitem{Grieser2006}
D.~Grieser.
\newblock The first eigenvalue of the {L}aplacian, isoperimetric constants, and
  the max flow min cut theorem.
\newblock {\em Arch. Math.}, 87(1):75--85, 2006.

\bibitem{IonLac2005}
I.~R. Ionescu and T.~Lachand-Robert.
\newblock Generalized cheeger sets related to landslides.
\newblock {\em Calculus of Variations and Partial Differential Equations},
  23(2):227--249, 2005.

\bibitem{KawFri2003}
B.~Kawohl and V.~Fridman.
\newblock Isoperimetric estimates for the first eigenvalue of the
  {$p$}-{L}aplace operator and the {C}heeger constant.
\newblock {\em Comment. Math. Univ. Carolin.}, 44(4):659--667, 2003.

\bibitem{KawLac2006}
B.~Kawohl and T.~Lachand-Robert.
\newblock {Characterization of Cheeger sets for convex subsets of the plane.}
\newblock {\em Pac. J. Math.}, 225(1):103--118, 2006.

\bibitem{KreKri2005}
D.~Krej{\v{c}}i{\v{r}}{\'{\i}}k and J.~K{\v{r}}{\'{\i}}z.
\newblock On the spectrum of curved planar waveguides.
\newblock {\em Publ. Res. Inst. Math. Sci.}, 41(3):757--791, 2005.

\bibitem{KrePra2011}
D.~Krej\v{c}i\v{r}{\'\i}k and A.~Pratelli.
\newblock {The Cheeger constant of curved strips.}
\newblock {\em Pac. J. Math.}, 254(2):309--333, 2011.

\bibitem{LefWei1997}
L.~Lefton and D.~Wei.
\newblock Numerical approximation of the first eigenpair of the
  {$p$}-{L}aplacian using finite elements and the penalty method.
\newblock {\em Numer. Funct. Anal. Optim.}, 18(3-4):389--399, 1997.

\bibitem{LeoPra2014}
G.~P. Leonardi and A.~Pratelli.
\newblock Cheeger sets in non-convex domains.
\newblock preprint ar{X}iv:1409.1376, 2014.

\bibitem{Miranda1964}
M.~Miranda.
\newblock Superfici cartesiane generalizzate ed insiemi di perimetro localmente
  finito sui prodotti cartesiani.
\newblock {\em Ann. Scuola Norm. Sup. Pisa (3)}, 18:515--542, 1964.

\bibitem{Parini_tesilaurea2006}
E.~Parini.
\newblock {\em Cheeger sets in the nonconvex case}.
\newblock Tesi di Laurea Magistrale, Universit{\`a} degli Studi di Milano,
  2006.

\bibitem{Parini_doktorarbeit2009}
E.~Parini.
\newblock {\em Asymptotic behaviour of higher eigenfunctions of the p-Laplacian
  as p goes to 1}.
\newblock PhD thesis, Universit{\"a}t zu K{\"o}ln, 2009.

\bibitem{RudOshFat1992}
L.~I. Rudin, S.~Osher, and E.~Fatemi.
\newblock Nonlinear total variation based noise removal algorithms.
\newblock {\em Physica D: Nonlinear Phenomena}, 60(1):259--268, 1992.

\bibitem{Steiner1840}
J.~Steiner.
\newblock {\"U}ber parallele fl{\"a}chen.
\newblock {\em Monatsber. Preuss. Akad. Wiss}, pages 114--118, 1840.

\bibitem{Strang2010}
G.~Strang.
\newblock Maximum flows and minimum cuts in the plane.
\newblock {\em J. Global Optim.}, 47(3):527--535, 2010.

\bibitem{StrZie1997}
E.~Stredulinsky and W.~P. Ziemer.
\newblock Area minimizing sets subject to a volume constraint in a convex set.
\newblock {\em J. Geom. Anal.}, 7(4):653--677, 1997.

\bibitem{Tamanini1982}
I.~Tamanini.
\newblock {\em Regularity results for almost minimal oriented hypersurfaces in
  {$R^N$}}.
\newblock Quaderni del Dipartimento di Matematica dell'Universit\`a di Lecce,
  1984.
\newblock available for download at http://cvgmt.sns.it/paper/1807/.

\bibitem{Weyl1939}
H.~Weyl.
\newblock On the {V}olume of {T}ubes.
\newblock {\em Amer. J. Math.}, 61(2):461--472, 1939.

\end{thebibliography}

\end{document}